\documentclass[a4paper,12pt,reqno]{amsart}

\usepackage{amsthm}
\usepackage{amsmath}
\usepackage{amssymb}
\usepackage{mathrsfs}
\usepackage{latexsym}
\usepackage{exscale}
\usepackage{geometry}
\usepackage{graphicx}
\usepackage[utf8]{inputenc}
\usepackage{verbatim}
\usepackage{enumerate}

\usepackage{color}
\definecolor{red}{rgb}{1,0.1,0.1}
\definecolor{blue}{rgb}{0.1,0.1,1}
\definecolor{vb}{RGB}{160,32,240}

\usepackage[pdftex]{hyperref}  

\theoremstyle{plain}

\newtheorem*{thm*}{Theorem}
\newtheorem*{prop*}{Proposition}

\numberwithin{equation}{section}
\newtheorem{thm}{Theorem}[section]
\newtheorem{lemma}[thm]{Lemma}
\newtheorem{corol}[thm]{Corollary}
\newtheorem{prop}[thm]{Proposition}

\theoremstyle{remark}
\newtheorem{remark}[thm]{Remark}

\theoremstyle{definition}

\newtheorem*{mydef*}{Definition}

\newtheorem{ejem}[thm]{Example}

\calclayout

\allowdisplaybreaks

\newcommand{\Z}{\mathbb{Z}}
\newcommand{\N}{\mathbb{N}}

\begin{document}

\title[Fractional Maximal operator in hyperbolic spaces]{Fractional Maximal operator in hyperbolic spaces}

\author[G.~Iba\~{n}ez~Firnkorn]{Gonzalo Iba\~{n}ez-Firnkorn}
\address{G.~Iba\~{n}ez~Firnkorn\\
Instituto de Matemática (INMABB), Departamento de Matemática, Universidad Nacional del Sur (UNS)-CONICET, Bahía Blanca, Argentina}
\email{gonzalo.ibanez@uns.edu.ar}

\author[E.~Ramadori]{Emanuel Ramadori}
\address{E.~Ramadori\\
Instituto de Matemática (INMABB), Departamento de Matemática, Universidad Nacional del Sur (UNS)-CONICET, Bahía Blanca, Argentina}
\email{emanuel.ramadori@uns.edu.ar}

\thanks{Both authors were  partially supported by
CONICET and  SGCyT-UNS}



\begin{abstract}
In this article, we introduce the fractional maximal operator on the Hyperbolic space, a non-doubling measure space, and study the weighted boundedness. Motivated in the weighted boundedness of Hardy-Littlewood maximal studied by Antezana and Ombrosi in \cite{AO23}, we give conditions for the weak type and strong type estimate for fractional maximal. Also, we provide examples of weights for which the fractional maximal operator satisface weak type $(p,q)$  inequality but strong type $(p,q)$ inequality fails.
\end{abstract}

\maketitle

\section{Introduction}
In harmonic analysis two important operators are the Hardy-Littlewood maximal operator and the fractional maximal operator. In  the context of a metric measure space $(X,\mu,d)$, it can be defined the centered Hardy-Littlewood maximal operator on balls as
$$M_{X}f(x)=\sup_{r>0}\frac1{\mu(B_{X}(x,r))}\int_{B_{X}(x,r)} |f(y)|d\mu(y).$$
This  maximal operator was widely studied in the literature. It is well known that if the measure $\mu$ is doubling, that is $\mu(B_X(x,2r))\leq C_X \mu(B_X(x,r))$, then $M_X$ is bounded on $L^p(X,w\,d\mu)$ with a certain class of weights $w$. In the case of $L^p(\mathbb{R}^n,w\,dx)$, where $dx$ is the Lebesgue measure, the class of weight is the classical Muckenhoupt weight $A_p$.

In the classical study of the boundedness of maximal operators the condition of doubling measure plays a {\it crucial} role. In recent works, the authors consider non-doubling measure spaces, in particular measure with a exponential growth. These spaces are for example the infinite rooted  $k$-ary tree studied in \cite{NT10,ORR22,ORRS21,GR23} and the $n$-dimensional hyperbolic space studied in \cite{AO23}. Also, in \cite{GRS23} the authors consider harmonic $NA$ groups.

In this paper we focus in the context of the  $n$-dimensional hyperbolic space, defined as follow:

Let $\mathcal{H}^n$ be the $n$-dimensional hyperbolic space, i.e. the unique (up to isometries) $n$-dimensional, complete and simply connected Riemannian manifold with constant sectional curvature $-1$. Let $\mu_n$ be the corresponding volume measure. If $B_{\mathcal{H}}(x,r)$ denotes the hyperbolic ball of  radius $r$ centered at $x$, the centered Hardy-Littlewood maximal operator on $\mathcal{H}^n$ is defined as 
$$Mf(x)=\sup_{r>0}\frac1{\mu_n(B_{\mathcal{H}}(x,r))}\int_{B_{\mathcal{H}}(x,r)} |f(y)|d\mu_n(y)$$
and the centered fractional maximal operator is defined, for $0<\alpha<n$, as
$$M_{\alpha}f(x)=\sup_{r>0}\frac{\mu_n(B_{\mathcal{H}}(x,r))^{\frac{\alpha}{n}}}{\mu_n(B_{\mathcal{H}}(x,r))}\int_{B_{\mathcal{H}}(x,r)} |f(y)|d\mu_n(y).$$

In the seminar work \cite{SW78}, Stein and Wainger proposed the study of the end-point estimates for the centered Hardy-Littlewood maximal function when the curvature of the underline space could be non-negative. In this more general scenario, the euclidean spaces $\mathbb{R}^n$ and the aforementioned hyperbolic spaces $\mathcal{H}^n$ represent two extreme cases. 

In \cite{S81}, Strömberg proved the (unweighted) weak type $(1,1)$ boundedness of $M$ in symmetric spaces of non-compact type, suggesting that the behavior of maximal operator is the same in both spaces, $\mathbb{R}^n$ and $\mathcal{H}^n$. However, this is not the case in general, and it will be reveled by analyzing weighted estimates.

In the recent paper of \cite{AO23}, Antezana and Ombrosi studied the Fefferman-Stein inequality for the Hardy-Littlewood maximal operator in the hyperbolic setting. Also, they give weight conditions for the weak and strong type estimates.

On the other hand, in \cite{GR23} Ghosh and Rela studied the weighted inequalities for fractional maximal operator on the infinite rooted $k$-ary tree. In \cite{LS22}, Levi and Santagati studied the boundedness between un-weighted Lorentz spaces of homogeneous tree.

Now,  for the un-weighted boundedness of the fractional maximal operator on the hiperbolic space, observe that $M_{\alpha}f(x)\lesssim \|f\|_{L^1(\mathcal{H}^n)}^{\frac{\alpha}{n}}(Mf(x))^{1-\frac{\alpha}{n}}$.
Since $M$ is of weak-type $(1,1)$ then $M_{\alpha}$ is bounded from $L^1(\mathcal{H}^n)$ into $L^{\frac1{1-\alpha}}(\mathcal{H}^n)$. On the other hand, using Hölder's inequality it is easy to see that $M_{\alpha}$ is bounded from $L^{1/\alpha}(\mathcal{H}^n)$ into $ L^{\infty}(\mathcal{H}^n)$. Interpolating these estimates, we obtain that the operator $M_{\alpha}$ is bounded from $L^{p}(\mathcal{H}^n)$ into $L^{q}(\mathcal{H}^n)$ where $\frac1{q}=\frac1{p}-\frac{\alpha}{n}$ and $1<p<\frac{n}{\alpha}$.

At this point, we need to recall some estimates of the measure in the hyperbolic space. In details in the hyperbolic setting,  there is an exponential growth of the measure, that is 
\begin{equation}
\mu_n(B_{\mathcal{H}}(x,r))=\Omega_n \int_{0}^r (\sinh t)^{n-1}\,dt \simeq_n \frac{r^n}{1+r^n}e^{(n-1)r},
\end{equation}
where $\Omega_n$ is the euclidean $(n-1)$-volume of the sphere $S^{n-1}$ and the subindex in the symbol $\simeq$ means that the constant behind this symbol depends only on dimension $n$. Another important estimates for this work is the measure of the intersection of two balls, that is 
\begin{equation}\label{eq: interBolas}
\mu_n(B_{\mathcal{H}}(x,r)\cap B_{\mathcal{H}}(y,s))\lesssim_n e^{(n-1)\frac{r+s-d_{\mathcal{H}^n}(x,y)}{2}}.
\end{equation}
This estimate is well known in the literature, for instance for geometric estimate we refer to \cite{CN07}, \cite{Rlibro} and \cite[Section 2]{AO23}.

Observe that this exponential behavior, as well as the metric properties of $\mathcal{H}^n$, make the classical covering arguments fail. 
To avoid this problem, inspired in the papers \cite{CS74,S81,AO23} we split the fractional maximal operator in a local part and a part far of the origin. In details, we consider the fractional average operator $A_{r,\alpha}$, defined for $0<\alpha<n$ and $r>0$ by
$$A_{r,\alpha}f(x)=\frac1{\mu_n(B_{\mathcal{H}}(x,r))^{1-\frac{\alpha}{n}}}\int_{B_{\mathcal{H}}(x,r)} |f(y)|d\mu_n(y).$$

Now, we can split the fractional maximal operator in a local fractional maximal operator $M_{\alpha}^{loc}$, where the supremum is taken over all radius less or equal than 2, and a fractional maximal operator $M_{\alpha}^{far}$ where the supremum is over all radius greater that 2. That is, 
\begin{equation}
M_{\alpha}f(x)\leq \sup_{r\leq 2} A_{r,\alpha}f(x)+\sup_{r> 2} A_{r,\alpha}f(x) =: M_{\alpha}^{loc}f(x) +M_{\alpha}^{far}f(x).
\end{equation}

Note that the operator $M_{\alpha}^{loc}$ behaves as in the Euclidean setting. The main difficulties appear in the estimation of $M_{\alpha}^{far}$.

For the local part, as in  the Euclidean setting, we need a local fractional class of weights, that is $A_{p,q}^{loc}(\mathcal{H}^n)$. We said that a weight $w$ belong to the class $A_{p,q}^{loc}(\mathcal{H}^n)$, $1<p,q<\infty$, if

\begin{equation}\label{eq: Apqloc}
\sup_{0<r(B)\leq 2} \left(\frac1{\mu_n(B)}\int_B w^q \,d\mu_n \right) \left(\frac1{\mu_n(B)}\int_B w^{-p'} \,d\mu_n \right)^{\frac{q}{p'}}<\infty
\end{equation}
where the supremum is taken over all balls $B$ with radius $r(B)\leq 2$. 
Also, a weight $w$ belong to the class $A_{1,q}^{loc}(\mathcal{H}^n)$, $1<q<\infty$, if
\begin{equation}\label{eq: A1qloc}
\sup_{0<r(B)\leq 2} \left(\frac1{\mu_n(B)}\int_B w^q \,d\mu_n \right) \|w^{-1}\chi_B\|_{\infty}^q<\infty
\end{equation}
where the supremum is taken over all balls $B$ with radius $r(B)\leq 2$.

The statement of our main result is 
\begin{thm}\label{AcotM}
Let $0<\alpha<n$, $1\leq p<\frac{n}{\alpha}$ and $\frac1{q}=\frac1{p}-\frac{\alpha}{n}$ and $w$ be a weight. Suppose that
\begin{enumerate}[\label=(i)]
\item \label{hip 1}$w\in A_{p,q}^{loc}(\mathcal{H}^n)$.
\item \label{hip 2}There exists $0<\beta<1$ and $\beta \leq \gamma <p$ such that for every $r\geq 1$ we have 
\begin{equation}\label{eq: condition Ar}
\int_F A_{r,\alpha}(\chi_E)(y)w^q(y)d\mu_n(y)\lesssim e^{(n-1)r(1-\frac{\alpha}{n})(\beta-1)} w^p(E)^{\frac{\gamma}{p}}w^q(F)^{1-\frac{\gamma}{q}},
\end{equation}
for any pair of measurable subsets $E,F\subset \mathcal{H}^n$.
\end{enumerate} 
Then, 
\begin{equation}\label{eq: debil Malpha}
\|M_{\alpha}f\|_{L^{q,\infty}(w^q)} \lesssim \|f\|_{L^p(w^p)}.
\end{equation}
Furthermore, if $\beta<\gamma$ then for each fixed $\delta\geq 0$ we have
\begin{equation}\label{eq: Acot prom}
\sum_{j=1}^{\infty} j^{\delta}\|A_jf\|_{L^q(w^q)}\lesssim \|f\|_{L^p(w^p)}.
\end{equation}
In particular, if $\beta<\gamma$ then
\begin{equation}\label{eq: fuerte Malpha}
\|M_{\alpha}f\|_{L^q(w^q)}\lesssim \|f\|_{L^p(w^p)}.
\end{equation}
\end{thm}

\begin{remark}
Observe that this theorem give us the boundedness of $M_{\alpha}$ from $L^p(\mathcal{H}^n,w^p)$ into $L^q(\mathcal{H}^n,w^q)$ for $1<p<\frac{n}{\alpha}$ and $\frac1{q}=\frac1{p}-\frac{\alpha}{n}$, this result is similar to the classical boundedness of $M_{\alpha}$ in the euclidean setting. There exist a weight that satisfies \eqref{hip 1} and \eqref{hip 2} but is not in $A_{p,q}(\mathcal{H}^n)$, see Example \ref{Ejemplos} (ii).
\end{remark}

Now, we centered to provide a more tractable condition that condition \eqref{eq: condition Ar}. For it, we study the behavior of the weight at each level and we need the following notation, 
\begin{align*}
 \mathcal{C}_1=B_\mathcal{H}(0,1)\setminus \{0\} && \text{ and } && \mathcal{C}_{j+1}=B_\mathcal{H}(0,j+1)\setminus B_{\mathcal{H}}(0,j) && \text{for each } j\in \N.
\end{align*}
In the following proposition we show an sufficient condition for the behavior of the weight at each level.

\begin{prop}\label{PropEx}
Let $0<\alpha<n$,  $1\leq p<\frac{n}{\alpha}$ and $\frac1{q}=\frac1{p}-\frac{\alpha}{n}$.
\label{Prop a)}If $w$ be a weight that satisfy the property that there exists $\delta<\min\left\{p, \frac{q}{p}+p-\frac{p}{1-\frac{\alpha}{n}}\right\}$, such that for all $j,l,r\in \N$ with $|l-j|\leq r$,
\begin{equation}\label{eq: Cond Cj}
w^q(\mathcal{C}_l\cap B_{\mathcal{H}}(x,r))\lesssim e^{(n-1)\frac{r+l-j}{2}(p-\delta)}e^{(n-1)r\delta}(w^p(x))^{\frac{q}{p}} \qquad \text{ for a.e. } x\in \mathcal{C}_j,
\end{equation} 
then  the condition \eqref{eq: condition Ar} holds for $\beta=\frac{p}{(1-\frac{\alpha}{n})(\frac{q}{p}+p-\delta)}$ and $\gamma=\frac{q}{\frac{q}{p}+p-\delta}$.

\end{prop}

Observe that the Proposition above and $A_{p,q}^{loc}$ give us conditions for the weak type and strong type estimates, since $\beta\leq \gamma$, and the equal holds if and only if $p=1$. On the other hand, it is possible to study condition in which only holds the weak type inequality. In details, if $w$ be a weight that satisfy the property that there exists $\delta<1$, such that for all $j,l,r\in \N$ with $|l-j|\leq r$,
\begin{equation}\label{eq: Cond Cj2}
w^q(\mathcal{C}_l\cap B_{\mathcal{H}}(x,r))\lesssim e^{(n-1)\frac{r+l-j}{2}(q(1-\frac{\alpha}{n})-\delta)}e^{(n-1)r\delta}(w^p(x))^{\frac{q}{p}} \qquad \text{ for a.e. } x\in \mathcal{C}_j,
\end{equation} 
then the condition \eqref{eq: condition Ar} holds for $\beta=\gamma=\frac{q}{q+1-\delta}$.
In particular, there exist weights that satisfies the condition \eqref{eq: Cond Cj2} but not the \eqref{eq: Cond Cj}. That give examples for weighted weak type of $M_{\alpha}$ in which not holds the strong type, see Example \ref{Ejemplos} (i).

The following Corollary give us conditions to generate examples for the boundedness of the fractional maximal operator.
\begin{corol}\label{Corollary}
Let $0<\alpha<n$, $1\leq p<\frac{n}{\alpha}$, $\frac1{q}=\frac1{p}-\frac{\alpha}{n}$ and $w\in A_{p,q}^{loc}(\mathcal{H}^n)$  be a weight such that there exists $\delta<1$, such that for all $j,l,r\in \N$ with $|l-j|\leq r$,
\begin{equation*}
w^{q}(\mathcal{C}_l\cap B_{\mathcal{H}}(x,r))\lesssim e^{(n-1)\frac{r+l-j}{2}(q(1-\frac{\alpha}{n}) -\delta)}e^{(n-1)r\delta}(w^{p}(x))^{\frac{q}{p}} \qquad \text{ for a.e. } x\in \mathcal{C}_j.
\end{equation*} 
Then, 
\begin{equation}
\|M_{\alpha}f\|_{L^{q,\infty}(w^{q})}\lesssim \|f\|_{L^{p}(w^{p})}.
\end{equation}
Furthermore, if there exists 
 $\xi <\min\left\{p, \frac{q}{p}+p-\frac{p}{1-\frac{\alpha}{n}}\right\} $, such that for all $j,l,r\in \N$ with $|l-j|\leq r$,
\begin{equation*}
w^{q}(\mathcal{C}_l\cap B_{\mathcal{H}}(x,r))\lesssim e^{(n-1)\frac{r+l-j}{2}(p -\xi)}e^{(n-1)r\xi}(w^{p}(x))^{\frac{q}{p}} \qquad \text{ for a.e. } x\in \mathcal{C}_j,
\end{equation*} 
we have
\begin{equation}\label{Corolario fuerte}
\|M_{\alpha}f\|_{L^{q}(w^{q})}\lesssim \|f\|_{L^p(w^{p})}.
\end{equation}
\end{corol}

The remainder of the paper is organized as follows. Section 2 is devoted to the proofs of the main results. Finally, in section 3 we show examples of weights that satisfies the hypothesis of Corollary \ref{Corollary} and examples in which the fractional maximal operator is bounded but the weight  does not belong in the classical fractional class of weights.

\section{Proof of main results}

\subsection{Proof of Theorem \ref{AcotM}}

For the proof of Theorem \ref{AcotM} we need the following Lemma

\begin{lemma}\label{Lemma}
Let $0<\alpha<n,$ $1\leq p,q<\infty$ and $w$ be a weight. Assume that there exist $0<\beta<1$ and $\beta\leq\gamma<p$ such that 
\begin{equation}\label{eq: cond Lemma}
\int_{F}A_{r,\alpha} (\chi_E)w^q\, d\mu_n 
\lesssim e^{(n-1)r(\beta-1)(1-\frac{\alpha}{n})}w^p(E)^{\frac{\gamma}{p}}w^q(F)^{1-\frac{\gamma}{q}},
 \end{equation}
for each $E$ and $F$ measurable subset of $\mathcal{H}^n$.

Then, for each $r\in \N$ and $\lambda >0$ we have
\begin{align*}
 w^q(\{A_{r,\alpha}(A_1f) >\lambda\})
 \lesssim &\sum_{k=0}^ {r} e^{(n-1)k\frac{q}{\gamma}} \left(\frac{e^{(n-1)r}}{e^{(n-1)k}}\right)^{\frac{q}{\gamma}\varepsilon} 
     e^{(n-1)r\frac{q}{\gamma}(\beta-1-\beta\frac{\alpha}{n})}
\\&\qquad \qquad \times w^p\left(\left\{A_{2}(f)e^{(n-1)r\frac{\alpha}{n}}\geq \eta e^{(n-1)(k-1)}\lambda\right\}\right)^\frac{q}{p}  
\end{align*}
      for all $\varepsilon\in (0,1).$
\end{lemma}
    
\begin{proof}
We can assume without lost of generality that $f\geq 0$ and we denote $f_1=A_1f$ and $f_2=A_2f$. 
We can split the function $f_1$ as follows
$$
f_{1}\leq 1 +\sum_{k=0}^{r} e^{(n-1)k}\chi_{E_k} +f_{1}\chi_{\{f_{1} \geq e^{(n-1)r}\}},
$$
where $E_{k}=\left\{ e^{(n-1)(k-1)}\leq f_1< e^{(n-1)k}\right\}$. Now, applying the operator $A_{r,\alpha}$ we obtain
$$
A_{r,\alpha}f_{1}\leq e^{(n-1)r\frac{\alpha}{n}}+\sum_{k=0}^{r} e^{(n-1)k} A_{r,\alpha}\left(\chi_{E_k}\right) +A_{r,\alpha}\left(f_{1}\chi_{\{f_{1} \geq e^{(n-1)r}\}}\right).
$$
Let $\Omega=\left\{ A_{r,\alpha} f_1 > 2e^{(n-1)r\frac{\alpha}{n}}\right\}$. Then, for any $x\in E$ we have
$$
\sum_{k=0}^{r} e^{(n-1)k} A_{r,\alpha}\left(\chi_{E_k}\right) +A_{r,\alpha}\left(f_{1}\chi_{\{f_{1} \geq e^{(n-1)r\frac{\alpha}{n}}\}}\right)>e^{(n-1)r\frac{\alpha}{n}}
$$

Consider the following sets
$$
I:=\left\{\sum_{k=0}^{r} e^{(n-1)k} A_{r,\alpha}\left(\dfrac{1}{e^{(n-1)r\frac{\alpha}{n}}}\chi_{E_k}\right)>\dfrac{1}{2}\right\}
$$
and 
$$
II:=\left\{A_{r,\alpha}\left(\dfrac{1}{e^{(n-1)r}}f_{1}\chi_{\{f_{1} \geq e^{(n-1)r\frac{\alpha}{n}}\}}\right)\neq 0\right\},
$$
so $w^q(\Omega)\leq w^q(I)+w^q(II).$

For the estimate of $w^q(I)$, we set $\varepsilon\in(0,1)$.  Observe that if 
$$
\sum_{k=0}^{r} e^{(n-1)k} A_{r,\alpha}\left(\dfrac{1}{e^{(n-1)r\frac{\alpha}{n}}}\chi_{E_k}\right)>\frac{1}{2},
$$
then we necessarily have for some $k_0\in\mathbb{N}$ such that $k_0\leq r$,
$$
A_{r,\alpha}\left(\dfrac{1}{e^{(n-1)r\frac{\alpha}{n}}}\chi_{E_{k_0}}\right) >  \dfrac{e^{(n-1)\varepsilon}-1}{e^{(n-1)(k_0+2)}}\left(\dfrac{e^{(n-1)k_0}}{e^{(n-1)r}}\right)^{\varepsilon}.
$$

Indeed, in other case 
\begin{align*}
\frac{1}{2}&
< \sum_{k=0}^{r} e^{(n-1)k} A_{r,\alpha}\left(\dfrac{1}{e^{(n-1)r\frac{\alpha}{n}}}\chi_{E_k}\right)
\\& \leq\dfrac{e^{(n-1)\varepsilon}-1}{e^{(n-1)\left(r\varepsilon+2\right)}}\dfrac{e^{(n-1)(r+1)\varepsilon}-1}{e^{(n-1)\varepsilon}-1}
<\dfrac{1}{e^{(n-1)}}
<\frac{1}{2},
\end{align*}
which leads a contradiction. Thus 
$$
w^q(I)\leq \sum_{k=0}^{r}
w^q(F_k),
$$
where
$$
F_k=\left\{A_{r,\alpha}\left(\chi_{E_{k}}\right)\geq \dfrac{e^{(n-1)\varepsilon}-1}{e^{(n-1)(k+2)}}e^{(n-1)r\frac{\alpha}{n}}\left(\dfrac{e^{(n-1)k}}{e^{(n-1)r}}\right)^{\varepsilon}\right\}.
$$

Observe that for each $k$,  $F_k$ has finite measure and 
$$
w^q(F_k) \dfrac{e^{(n-1)\varepsilon}-1}{e^{(n-1)(k+2)}}e^{(n-1)r\frac{\alpha}{n}}\left(\dfrac{e^{(n-1)k}}{e^{(n-1)r}}\right)^{\varepsilon}
\leq \int_{F_k} A_{r,\alpha}(\chi_{E_k}) w^q\, d\mu_n .
$$
Using the condition \eqref{eq: cond Lemma}, we obtain 
\begin{align*}
w^q(F_k)&
\lesssim \left[\dfrac{e^{(n-1)(k+2)}}{e^{(n-1)\varepsilon}-1} \dfrac{1}{e^{(n-1)r\frac{\alpha}{n}}} \left(\dfrac{e^{(n-1)r}}{e^{(n-1)k}}\right)^\varepsilon e^{(n-1)r(\beta-1)(1-\frac{\alpha}{n})}\right]^\frac{q}{\gamma} w^p(E_k)^\frac{q}{p}
\\ & \lesssim e^{(n-1)k\frac{q}{\gamma}}\left(\dfrac{e^{(n-1)r}}{e^{(n-1)k}}\right)^{\frac{q}{\gamma}\varepsilon}e^{(n-1)r\frac{q}{\gamma}(\beta-1-\beta\frac{\alpha}{n})}w^p(E_k)^\frac{q}{p}.    
\end{align*}

Then, 
\begin{equation}\label{eq: I}
w^q(I)\lesssim 
\sum_{k=0}^{r} e^{(n-1)k\frac{q}{\gamma}}\left(\dfrac{e^{(n-1)r}}{e^{(n-1)k}}\right)^{\frac{q}{\gamma}\varepsilon}e^{(n-1)r\frac{q}{\gamma}(\beta-1-\beta\frac{\alpha}{n})}
w^p\left(\left\{f_1\geq e^{(n-1)(k-1)}\right\}\right)^{\frac{q}{p}}.
\end{equation}

Now, we estimate $w^{q}(II)$. First, observe that 

\begin{align*}
w^q(II)
&=w^q\left(\left\{x\in\mathcal{H}^n : A_{r,\alpha}\left(f_{1}\chi_{\{f_{1} \geq e^{(n-1)r}\}}\right)(x)\neq 0\right\}\right)
\\&\leq w^q\left(\left\{x\in\mathcal{H}^n: B(x,r)\cap \left\{f_{1}\geq e^{(n-1)r}\right\}\neq \emptyset\right\}\right).
\end{align*}

Let $x\in\mathcal{H}^n$ such that $B_\mathcal{H}(x,r)\cap \left\{f_{1}\geq e^{(n-1)r}\right\}\neq \emptyset$, and $y\in B_\mathcal{H}(x,r)\cap \left\{f_{1}\geq e^{(n-1)r}\right\}$. Then, we have 
$$
B_\mathcal{H}(y,1)\subset B_\mathcal{H}(x,r+1)\cap\{f_2\geq c_0 e^{(n-1)r}\},
$$
where $c_0=\frac{\mu_n\left(B_\mathcal{H}(0,1)\right)}{\mu_n\left(B_\mathcal{H}(0,2)\right)}$.

Then, 
\begin{align*}
 A_{r+1,\alpha}\left(\chi_{\left\{f_2\geq c_0e^{(n-1)r}\right\}}\right)(x)&\geq \dfrac{\mu_n(B_\mathcal{H}(y,1))}{\mu_n(B_\mathcal{H}(x,r+1))^{(1-\frac{\alpha}{n})}}
= \dfrac{e^{(n-1)\frac{\alpha}{n}}}{c_1e^{(n-1)r(1-\frac{\alpha}{n})}}.
\end{align*}

Now, considering the sets
\begin{align*}
F=\left\{ A_{r+1,\alpha}\left(\chi_{\left\{f_2\geq c_0e^{(n-1)r}\right\}}\right)(x)\geq \dfrac{e^{(n-1)\frac{\alpha}{n}}}{c_1e^{(n-1)r(1-\frac{\alpha}{n})}}\right\},
&&
E=\left\{f_2\geq c_0e^{(n-1)r}\right\}
\end{align*}
and using the condition \eqref{eq: cond Lemma} we obtain
\begin{align*}
w^q\left(F\right) 
&\leq c_1\frac{e^{(n-1)r(1-\frac{\alpha}{n})}}{e^{(n-1)\frac{\alpha}{n}}}\int_{F} A_{r+1,\alpha}(\chi_E) w^q\, d\mu_n
\\
&\lesssim c_1\frac{e^{(n-1)r(1-\frac{\alpha}{n})}}{e^{(n-1)\frac{\alpha}{n}}}e^{(n-1)r(\beta-1)(1-\frac{\alpha}{n})}w^p(E)^{\frac{\gamma}{p}}w^q(F)^{1-\frac{\gamma}{q}},
\end{align*}
and this imply  
\begin{equation}\label{eq: II}
w^q(II)
\leq w^q(F) 
\lesssim e^{(n-1)r\beta(1-\frac{\alpha}{n})\frac{q}{\gamma}} w^p\left(\left\{ f_2\geq c_0 e^{(n-1)r}\right\}\right)^\frac{q}{p}.
\end{equation}

Taking account \eqref{eq: I} and \eqref{eq: II} we obtain 
\begin{align*}
&w^q\left(\left\{A_{r,\alpha} f_{1} >2e^{(n-1)r}\right\}\right)
\leq w^q(I)+w^q(II)
\\&\lesssim\sum_{k=0}^{r} e^{(n-1)k\frac{q}{\gamma}}\left(\dfrac{e^{(n-1)r}}{e^{(n-1)k}}\right)^{\frac{q}{\gamma}\varepsilon}e^{(n-1)r\frac{q}{\gamma}(\beta-1-\beta\frac{\alpha}{n})}
w^p\left(\left\{f_1\geq e^{(n-1)(k-1)}\right\}\right)^{\frac{q}{p}}
\\&\qquad+ e^{(n-1)r\beta(1-\frac{\alpha}{n})\frac{q}{\gamma}} 
w^p\left(\left\{ f_2\geq c_0 e^{(n-1)r}\right\}\right)^\frac{q}{p}.
\end{align*}

To complete the proof, note that the second term of the estimate above is controlled by the last term in the sum. Using homogeneity, we obtain  
\begin{align*}
w^q(\{A_{r,\alpha} f_1 >\lambda\})
&\lesssim 
\sum_{k=0}^ {r} e^{(n-1)k\frac{q}{\gamma}} \left(\frac{e^{(n-1)r}}{e^{(n-1)k}}\right)^{\frac{q}{\gamma}\varepsilon} 
     e^{(n-1)r\frac{q}{\gamma}(\beta-1-\beta\frac{\alpha}{n})}
\\&\qquad \qquad \times w^p\left(\left\{A_{2}(f)e^{(n-1)r\frac{\alpha}{n}}\geq \eta e^{(n-1)(k-1)}\lambda\right\}\right)^\frac{q}{p} .
\end{align*}
 \end{proof}

The distributional estimate obtained in the previous lemma will allow us to prove the Theorem \ref{AcotM}.
 \begin{proof}[Proof of Theorem \ref{AcotM}]
 As we discussed in the introduction, $M_\alpha^{loc}$ behaves as in the Euclidean case, so by the condition \textit{(\ref{hip 1})}, that is, $w\in A_{p,q}^{loc}(\mathcal{H}^n)$, we have that for $p\geq 1$
 $$
 \left\Vert M_\alpha^{loc}f\right\Vert_{L^{q,\infty}(w^q)}\lesssim \Vert f\Vert_{L^p(w^p)},
 $$
and for $p>1$
$$\left\Vert M_\alpha^{loc}f\right\Vert_{L^{q}(w^q)}\lesssim \Vert f\Vert_{L^p(w^p)}.$$
Therefore, we focus on the estimation of $M_\alpha^{far}$. Let start with  the extreme case  $\beta=\gamma$. It follows from Lemma \ref{Lemma} that
 
 \begin{align*}
 &w^q\left(\left\{M_\alpha^{far}f>\lambda\right\}\right)
 \leq  w^q\left(\left\{M_\alpha^{far}(A_1f)>\lambda\right\}\right)
 \\&\leq \sum_{r=1}^\infty w^q \left(\left\{A_{r,\alpha}(A_1f)>\lambda\right\}\right)
 \\&\lesssim \sum_{r=0}^\infty \sum_{k=0}^ {r} e^{(n-1)k\frac{q}{\gamma}} \left(\frac{e^{(n-1)r}}{e^{(n-1)k}}\right)^{\frac{q}{\gamma}\varepsilon} 
     e^{(n-1)r\frac{q}{\gamma}(\beta-1-\beta\frac{\alpha}{n})}
 w^p\left(\left\{A_{2}(f)e^{(n-1)r\frac{\alpha}{n}}\geq \eta e^{(n-1)(k-1)}\lambda\right\}\right)^\frac{q}{p}  
 \\& =\sum_{r=0}^\infty \sum_{k=0}^ {r} e^{(n-1)k\frac{q}{\beta}} \left(\frac{e^{(n-1)r}}{e^{(n-1)k}}\right)^{\frac{q}{\beta}\varepsilon} 
     e^{(n-1)r\frac{q}{\beta}(\beta-1-\beta\frac{\alpha}{n})}\left(\int_{\mathcal{H}^n}\chi_{\left\{A_2(f)\geq \eta e^{(n-1)(k-1-r\frac{\alpha}{n})}
     \lambda\right\}}w^p\, d\mu_n\right)^\frac{q}{p}.
 \end{align*}
 
By making the change of variable $l=k-\lceil r\frac{\alpha}{n}\rceil$, where $\lceil z \rceil$ is the ceiling of the number $z$, the last estimate is controlled  by

 \begin{align*}
     &\sum_{r=0}^\infty \sum_{l=-\lceil r\frac{\alpha}{n}\rceil}^ {r-\lceil r\frac{\alpha}{n}\rceil} e^{(n-1)(l+\lceil r\frac{\alpha}{n}\rceil)\frac{q}{\beta}} \left(\frac{e^{(n-1)r}}{e^{(n-1)(l+\lceil r\frac{\alpha}{n}\rceil)}}\right)^{\frac{q}{\beta}\varepsilon} 
     e^{(n-1)r\frac{q}{\beta}(\beta-1-\beta\frac{\alpha}{n})}\left(\int_{\mathcal{H}^n}\chi_{\left\{A_2(f)\geq \eta e^{(n-1)(l-2)}
     \lambda\right\}}w^p\, d\mu_n\right)^\frac{q}{p}
     \\&\lesssim \left(\int_{\mathcal{H}^n} \sum_{r=0}^\infty \sum_{l=-\lceil r\frac{\alpha}{n}\rceil}^ {-1} e^{(n-1)l\frac{p}{\beta}(1-\varepsilon)} 
     e^{(n-1)r\frac{p}{\beta}(\beta-\beta\frac{\alpha}{n}+\varepsilon-\varepsilon\frac{\alpha}{n}+\frac{\alpha}{n}-1)}\chi_{\left\{A_2(f)\geq \eta e^{(n-1)(l-2)}
     \lambda\right\}}w^p\, d\mu_n\right)^\frac{q}{p} 
     \\&+ \left(\int_{\mathcal{H}^n} \sum_{r=0}^\infty \sum_{l=0}^{r-\lceil r\frac{\alpha}{n}\rceil} e^{(n-1)l\frac{p}{\beta}(1-\varepsilon)} 
     e^{(n-1)r\frac{p}{\beta}(\beta-\beta\frac{\alpha}{n}+\varepsilon-\varepsilon\frac{\alpha}{n}+\frac{\alpha}{n}-1)}\chi_{\left\{A_2(f)\geq \eta e^{(n-1)(l-2)}
     \lambda\right\}}w^p\, d\mu_n\right)^\frac{q}{p} 
     \\&=I+II.
 \end{align*}

If we take $0<\varepsilon<1-\beta$, we can estimate $I$ and $II$ as follows
\begin{align*}
&I\lesssim\left(\int_{\mathcal{H}^n}\sum_{r=0}^\infty \sum_{l=-\infty}^ {-1} e^{(n-1)l\frac{p}{\beta}(1-\varepsilon-\beta)}e^{(n-1)r\frac{p}{\beta}(\beta-\beta\frac{\alpha}{n}+\varepsilon-\varepsilon\frac{\alpha}{n}-1+\frac{\alpha}{n})}\left(\dfrac{A_2(f)}{\eta\lambda}\right)^{p}w^p\, d\mu_n\right)^\frac{q}{p} 
     \\&\quad\lesssim \dfrac{1}{\lambda^q}\left(\int_{\mathcal{H}^n}\sum_{r=0}^\infty e^{-(n-1)r\frac{p}{\beta}(1-\beta-\varepsilon)(1-\frac{\alpha}{n})} A_2(f)^pw^p\, d\mu_n\right)^\frac{q}{p} 
     \\&\quad \lesssim \dfrac{1}{\lambda^q}\Vert A_2(f)\Vert_{L^p(w^p)}^q,
\end{align*}

and 

\begin{align*}
&II\lesssim\left(\int_{\mathcal{H}^n}\sum_{r=0}^\infty \sum_{l=0}^ {r-\lceil r\frac{\alpha}{n}\rceil}e^{(n-1)l\frac{p}{\beta}(1-\varepsilon-\beta)} 
     e^{(n-1)r\frac{p}{\beta}(\beta-\beta\frac{\alpha}{n}+\varepsilon-\varepsilon\frac{\alpha}{n}+\frac{\alpha}{n}-1)}e^{(n-1)lp}\chi_{\left\{A_2(f)\geq \eta e^{(n-1)(l-2)}
     \lambda\right\}}w^p\, d\mu_n\right)^\frac{q}{p}    \\&\quad\lesssim\left(\int_{\mathcal{H}^n}\sum_{l=0}^\infty \sum_{r=\lceil\frac{n}{n-\alpha}l\rceil}^ \infty \left(\dfrac{e^{(n-1)l}}{e^{(n-1)r(1-\frac{\alpha}{n})}}\right)^{1-\beta-\varepsilon} 
    e^{(n-1)lp}\chi_{\left\{A_2(f)\geq \eta e^{(n-1)(l-2)}
     \lambda\right\}}w^p\, d\mu_n\right)^\frac{q}{p}    \\&\quad\lesssim\left(\int_{\mathcal{H}^n}\sum_{l=0}^\infty 
    e^{(n-1)lp}\chi_{\left\{A_2(f)\geq \eta e^{(n-1)(l-2)}
     \lambda\right\}}w^p\, d\mu_n\right)^\frac{q}{p}
     \\&\quad\lesssim \left(\int_{\mathcal{H}^n}\left(\dfrac{A_2(f)}{\lambda}\right)^pw^p\, d\mu_n\right)^\frac{q}{p}
     \\&\quad = \dfrac{1}{\lambda^q}\Vert A_2(f)\Vert_{L^p(w^p)}^q.
\end{align*}

Thus, combining the two previous estimates, it follows that
$$
 w^q\left(\left\{M_\alpha^{far}(f)>\lambda\right\}\right)\lesssim \dfrac{1}{\lambda^q}\Vert A_2(f)\Vert _{L^p(w^p)}^q\lesssim  \dfrac{1}{\lambda^q}\Vert f\Vert _{L^p(w^p)}^q,
$$
where in the last inequality, we use that $w\in A_{p,q}^{loc}(\mathcal{H}^n)$ implies $w^p\in A_p^{loc}(\mathcal{H}^n)$.\\

Now, suppose that $\beta<\gamma$. Then, by the layer-cake formula and Lemma \ref{Lemma}, it follows that

 \begin{align*}
    &\Vert A_{r,\alpha} f\Vert_{L^q(w^q)}^q = q\int_0^\infty \lambda^{q-1}w^q(A_{r,\alpha}(f)> \lambda)\, d\lambda
    \\&\quad \lesssim\int_0^\infty \lambda^{q-1} \sum_{k=0}^ {r} e^{(n-1)k\frac{q}{\gamma}} \left(\frac{e^{(n-1)r}}{e^{(n-1)k}}\right)^{\frac{q}{\gamma}\varepsilon} 
     e^{(n-1)r\frac{q}{\gamma}(\beta-1-\beta\frac{\alpha}{n})}
 w^p\left(\left\{A_{2}(f)e^{(n-1)r\frac{\alpha}{n}}\geq \eta e^{(n-1)(k-1)}\lambda\right\}\right)^\frac{q}{p} \, d\lambda
    \\&\quad\lesssim e^{(n-1)r\frac{q}{\gamma}(\beta-1-\beta\frac{\alpha}{n}+\varepsilon)}\sum_{k=0}^ {r} e^{(n-1)k\frac{q}{\gamma}(1-\varepsilon)}
      \int_0^\infty \lambda ^{q-1}\left(\int_{\mathcal{H}^n}\chi_{\left\{A_2(f) e^{(n-1)r\frac{\alpha}{n}}\geq \eta e^{(n-1)(k-1)}\lambda\right\}}w^p\, d\mu_n\right)^\frac{q}{p} \, d\lambda
    \\&\quad\lesssim e^{(n-1)r\frac{q}{\gamma}(\beta-1-\beta\frac{\alpha}{n}+\varepsilon)}\sum_{k=0}^ {r} e^{(n-1)k\frac{q}{\gamma}(1-\varepsilon)}
        \left[\int_{\mathcal{H}^n} \left(\int_0^\infty\chi_{\left\{A_2(f) e^{(n-1)r\frac{\alpha}{n}}\geq \eta e^{(n-1)(k-1)}\lambda\right\}}\lambda^{q-1} d\lambda\right)^\frac{p}{q}w^p\, d\mu_n\right]^\frac{q}{p}                                                    \\&\quad\lesssim e^{(n-1)r\frac{q}{\gamma}(\beta-1-\beta\frac{\alpha}{n}+\varepsilon)}\sum_{k=0}^ {r} e^{(n-1)k\frac{q}{\gamma}(1-\varepsilon)}\left[\int_{\mathcal{H}^n} \left(\dfrac{A_2(f) e^{(n-1)r\frac{\alpha}{n}}}{\eta e^{(n-1)(k-1)}}\right)^p w^p\, d\mu_n\right]^\frac{q}{p}                                     \\&\quad\lesssim e^{(n-1)r\frac{q}{\gamma}(\beta-1-\beta\frac{\alpha}{n}+\varepsilon+\gamma\frac{\alpha}{n})}\sum_{k=0}^ {r} e^{(n-1)k\frac{q}{\gamma}(1-\varepsilon-\gamma)}\left[\int_{\mathcal{H}^n} A_2(f)^p w^p\, d\mu_n\right]^\frac{q}{p}  \\&\quad\lesssim e^{(n-1)r\frac{q}{\gamma}(\beta-1-\beta\frac{\alpha}{n}+\varepsilon+\gamma\frac{\alpha}{n})}e^{(n-1)r\frac{q}{\gamma}(1-\varepsilon-\gamma)}\Vert A_2(f)\Vert_{L^p(w^p)}^q 
        \\&\quad =e^{(n-1)r\frac{q}{\gamma}(\beta-\gamma)(1-\frac{\alpha}{n})}\Vert A_2(f)\Vert_{L^p(w^p)}^q 
 \end{align*}
 Since $w^p\in A_p^{loc}(\mathcal{H}^n)$ and $\beta-\gamma<0$, we get
 
 $$
\sum_{r=1}^{\infty} r^{\delta}\|A_{r,\alpha}(f)\|_{L^q(w^q)}\lesssim \sum_{r=1}^{\infty} r^{\delta}e^{(n-1)r\frac{1}{\gamma}(\beta-\gamma)(1-\frac{\alpha}{n})}\Vert A_2(f)\Vert_{L^p(w^p)}\lesssim \|f\|_{L^p(w^p)}.
$$
Hence \eqref{eq: Acot prom} holds. From \eqref{eq: Acot prom} and the fact that
$$
M_\alpha^{far}f\lesssim \sum_{r=1}^\infty A_{r,\alpha}(f),  
$$
 we obtain \eqref{eq: fuerte Malpha} and the proof of the Theorem \ref{AcotM} is complete.
\end{proof}
\subsection{Proof of Proposition \ref{PropEx}}
 
In this section, we provide a proof of Proposition \ref{PropEx} using the approach outlined in \cite{AO23}.
\begin{proof}
Let $E,F\subset \mathcal{H}^n$ be measurable sets. We denote
$$
 I=\int_{F}A_{r,\alpha}(\chi_{E})(x)w^q(x)\,d\mu_n(x)
$$
and split
 $I$ as follows
$$
  I_{l,j}=\int_{F_l}A_{r,\alpha}(\chi_{E_j})(x)w^q(x)\,d\mu_n(x),
$$
where $F_l=F\cap\mathcal{C}_l$ and $E_j=E\cap\mathcal{C}_j$, so that 
$I\lesssim\sum_{l,j\in\mathbb{N}}I_{l,j}.
$

Now, by the H\"older's and Minkowski's inequalities, and using the condition \eqref{eq: Cond Cj}, we have that
\begin{align}\label{eq: ineq 1}
I_{l,j}
&\simeq\frac1{e^{(n-1)r(1-\frac{\alpha}{n})}} \int_{F_l}\left(\int_{B_\mathcal{H}(y,r)}\chi_{E_j}(x)\,d\mu_n(x) \right) w^q(y)\,d\mu_n(y)
\nonumber
\\&\leq \frac1{e^{(n-1)r(1-\frac{\alpha}{n})}} w^q(F_l)^{1-\frac{p}{q}} \left[\int_{F_l}\left(\int_{E_j}\chi_{B_\mathcal{H}(y,r)}(x)\,d\mu_n(x) \right)^{\frac{q}{p}} w^q(y)\,d\mu_n(y)\right]^{\frac{p}{q}}
\nonumber
\\&\leq \frac1{e^{(n-1)r(1-\frac{\alpha}{n})}} w^q(F_l)^{1-\frac{p}{q}} \int_{E_j}\left(\int_{F_l}\chi_{B_\mathcal{H}(x,r)}(y)w^q(y)\,d\mu_n(y) \right)^{\frac{p}{q}} \,d\mu_n(x)
\nonumber
\\&\lesssim \frac1{e^{(n-1)r(1-\frac{\alpha}{n})}} w^q(F_l)^{1-\frac{p}{q}} \int_{E_j} \left(e^{(n-1)\frac{r+l-j}{2}(p-\delta)}e^{(n-1)r\delta} w^q(x)\right)^{\frac{p}{q}}\,d\mu_n(x)
\nonumber
\\&\leq \frac1{e^{(n-1)r(1-\frac{\alpha}{n})}}e^{(n-1)\frac{r+l-j}{2}(p-\delta)\frac{p}{q}}e^{(n-1)r\delta\frac{p}{q}} w^q(F)^{1-\frac{p}{q}}w^p(E_j).
\end{align}

On the other hand, using  the estimation \eqref{eq: interBolas}, we obtain that
 for any $y\in \mathcal{C}_l$ 
$$\mu_n(\mathcal{C}_j\cap B_\mathcal{H}(y,r))\lesssim e^{(n-1)\frac{r+j-l}{2}},$$
 and therefore
\begin{align}\label{eq: ineq 2}
I_{l,j}
&\simeq\frac1{e^{(n-1)r(1-\frac{\alpha}{n})}} \int_{F_l}\left(\int_{B_\mathcal{H}(y,r)}\chi_{E_j}(x)\,d\mu(x) \right) w^q(y)\,d\mu_n(y)
\nonumber
\\&\lesssim \frac1{e^{(n-1)r(1-\frac{\alpha}{n})}} \int_{F_l}e^{(n-1)\frac{r+j-l}{2}} w^q(y)\,d\mu_n(y)
\nonumber
\\&\leq \frac1{e^{(n-1)r(1-\frac{\alpha}{n})}} e^{(n-1)\frac{r+j-l}{2}} w^q(F_l).
\end{align}

Taking account the estimates \eqref{eq: ineq 1} and \eqref{eq: ineq 2},  we obtain
\begin{equation}\label{eq: min}
I
\lesssim \frac1{e^{(n-1)r(1-\frac{\alpha}{n})}}\sum_{l,j\in \N} 
\min\left\{e^{(n-1)\frac{r+l-j}{2}(p-\delta)\frac{p}{q}}e^{(n-1)r\delta\frac{p}{q}} w^q(F)^{1-\frac{p}{q}}w^p(E_j),
 e^{(n-1)\frac{r+j-l}{2}} w^q(F_l)
\right\}.
\end{equation}

Now, we consider the following sequences $\{A_j\}_{j\in \Z}$ and $\{B_l\}_{l\in \Z}$ defined as: 
\begin{align*}
A_j=\begin{cases}
w^p(E_j) &\text{ if } j>0
\\
0 &\text{ othewise}
\end{cases},
&&
B_l=\begin{cases}
w^q(F_l) &\text{ if } l>0
\\
0 &\text{  othewise }
\end{cases}.
\end{align*}
Also, set $A=w^p(E)$ and $B=w^q(F)$.
Let $s$ be a real parameter to be choosen later, and argue as follows 
\begin{align}\label{eq: s}
\sum_{l,j\in \N}&\min\left\{e^{(n-1)\frac{r+l-j}{2}(p-\delta)\frac{p}{q}}e^{(n-1)r\delta\frac{p}{q}} B^{1-\frac{p}{q}}A_j,
 e^{(n-1)\frac{r+j-l}{2}} B_l
\right\}\nonumber
\\&\leq \sum_{j\in \N} \sum_{l<j+s} e^{(n-1)\frac{r+l-j}{2}(p-\delta)\frac{p}{q}}e^{(n-1)r\delta\frac{p}{q}} B^{1-\frac{p}{q}}A_j \nonumber
+ \sum_{l\in \N} \sum_{j\leq l-s}e^{(n-1)\frac{r+j-l}{2}} B_l\nonumber
\\&\leq e^{(n-1)\frac{r}{2}\frac{p}{q}(p-\delta+2\delta)}B^{1-\frac{p}{q}}  \sum_{j\in \N}  \sum_{l<j+s} e^{(n-1)\frac{l-j}{2}\frac{p}{q}(p-\delta)}A_j \nonumber
+ e^{(n-1)\frac{r}{2}}\sum_{l\in \N} \sum_{j\leq l-s}e^{(n-1)\frac{j-l}{2}} B_l\nonumber
\\&\lesssim e^{(n-1)\frac{r}{2}\frac{p}{q}(p+\delta)}B^{1-\frac{p}{q}} \sum_{j\in \N}  e^{(n-1)\frac{s}{2}\frac{p}{q}(p-\delta)}A_j \nonumber
 + e^{(n-1)\frac{r}{2}}\sum_{l\in \N}e^{-(n-1)\frac{s}{2}} B_l\nonumber
\\& =  e^{(n-1)\frac{r}{2}\frac{p}{q}(p+\delta)} e^{(n-1)\frac{s}{2}\frac{p}{q}(p-\delta)} B^{1-\frac{p}{q}}  A
+e^{(n-1)\frac{r}{2}}e^{-(n-1)\frac{s}{2}} B.
\end{align}

In order to optimize the last expression, 
 let  $\varphi_{A,B}$ be the function defined by
$$\varphi_{A,B}(s)=
e^{(n-1)\frac{r}{2}\frac{p}{q}(p+\delta)} e^{(n-1)\frac{s}{2}\frac{p}{q}(p-\delta)} B^{1-\frac{p}{q}}  A
+e^{(n-1)\frac{r}{2}}e^{-(n-1)\frac{s}{2}} B.
$$

This function reaches its  absolute minimum at
$$
s_0=\log\left( \left[\frac{q}{(p-\delta)p}\frac{B^{\frac{p}{q}}}{A}\right]^{\frac{2}{n-1}\frac1{(p-\delta)\frac{p}{q}+1}}  \right) 
-r\frac{(p+\delta)\frac{p}{q}-1}{(p-\delta)\frac{p}{q}+1},
$$
and

$$
\varphi_{A,B}(s_0)=C_{p,q,\delta} e^{(n-1)\frac{r}{2}} 
 e^{\frac{(n-1)}{2}r\frac{(p+\delta)\frac{p}{q}-1}{(p-\delta)\frac{p}{q}+1}} A^{\frac1{(p-\delta)\frac{p}{q}+1}}B^{1-\frac{p}{q[(p-\delta)\frac{p}{q}+1]}}.
$$
Then, 
 we obtain that
\begin{align}\label{eq: phi}
\sum_{l,j\in \N}\min&\left\{e^{(n-1)\frac{r+l-j}{2}(p-\delta)\frac{p}{q}}e^{(n-1)r\delta\frac{p}{q}} w^q(F)^{1-\frac{p}{q}}w^p(E_j),
 e^{(n-1)\frac{r+j-l}{2}} w^q(F_l)
\right\}
\lesssim
\phi_{A,B}(s).
\end{align}
Hence,  \eqref{eq: min} and \eqref{eq: phi} implies that  
\begin{align*}
I
&\lesssim \frac{1}{e^{(n-1)r(1-\frac{\alpha}{n})}}\varphi_{A,B}(s_0)
\\&\lesssim  \frac{1}{e^{(n-1)r(1-\frac{\alpha}{n})}} e^{(n-1)\frac{r}{2}}
e^{\frac{(n-1)}{2}r\frac{(p+\delta)\frac{p}{q}-1}{(p-\delta)\frac{p}{q}+1}} w^p(E)^{\frac1{(p-\delta)\frac{p}{q}+1}}w^q(F)^{1-\frac{p}{q[(p-\delta)\frac{p}{q}+1]}}
\\&= \dfrac{1}{e^{(n-1)r(1-\frac{\alpha}{n})}}e^{(n-1)r\frac{p}{p-\delta+\frac{q}{p}}}w^p(E)^\frac{1}{(p-\delta)\frac{p}{q}+1}w^q(F)^{1-\frac{1}{p-\delta+\frac{q}{p}}}.
\end{align*}
 Then, the condition \eqref{eq: condition Ar} holds for $\beta=\frac{p}{(1-\frac{\alpha}{n})(\frac{q}{p}+p-\delta)}$ and $\gamma=\frac{q}{\frac{q}{p}+p-\delta}$, this complete the proof. 
\end{proof}
\section{Examples}

In this section we show examples of weights that satisfies our conditions.\\

Let $1\leq p <\frac{n}{\alpha}$ and $\frac1{q}=\frac1{p}-\frac{\alpha}{n}$. Consider $\theta\in \mathbb{R}$, we define the weight
$$w_{\theta}(x)=\frac1{\left[1+\mu
_n(B_\mathcal{H}(0,d_{\mathcal{H}^n}(0,x)))\right]^{\frac{\theta}{q}}}.$$

It is easy to see that $w_{\theta}\in A_{p,q}^{loc}(\mathcal{H}^n)$.
 
Observe, if $1-p\leq \theta<\min\left\{p, p+\frac{q}{p}-\frac{p}{1-\frac{\alpha}{n}}\right\} $, then $w_{\theta}$ satisfies the inequality \eqref{eq: Cond Cj} that is, 
$$w^q(\mathcal{C}_l\cap B_\mathcal{H}(x,r))\lesssim e^{(n-1)\frac{r+l-j}{2}(p-\delta)}e^{(n-1)r\delta}(w^p(x))^{\frac{q}{p}} \qquad \text{ for a.e. } x\in \mathcal{C}_j$$
with $\delta= \theta$. In particular, we can consider $\theta=\delta=1-p$.

On the other hand, if $-\frac{q}{p'}\leq \theta<1$, $w_{\theta}$ satisfies the estimate \eqref{eq: Cond Cj2}, 
 that is
$$w^q(\mathcal{C}_l\cap B_\mathcal{H}(x,r))\lesssim e^{(n-1)\frac{r+l-j}{2}(q(1-\frac{\alpha}{n})-\delta)}e^{(n-1)r\delta}(w^p(x))^{\frac{q}{p}} \qquad \text{ for a.e. } x\in \mathcal{C}_j$$
with $\delta=\theta$. In particular, if $\theta=-\frac{q}{p'}$ then $w_{\theta}(x)=\left[1+\mu_n(B_\mathcal{H}(0,d_{\mathcal{H}^n}(0,x)))\right]^{\frac1{p'}}$.\\

Now, we present examples for illustrate several point previously mentioned. Also, we show the difference between the classical $A_{p,q}$ class and our conditions for the weights.
\begin{ejem} 
\

\begin{enumerate}[(i)]\label{Ejemplos}
\item  Let $1< p <\frac{n}{\alpha}$ and $\frac1{q}=\frac1{p}-\frac{\alpha}{n}$. If we consider $\theta=-\frac{q}{p'}$ then we have that the weight $w_{\theta}$ satisfies the hypothesis of Corollary \ref{Corollary}, so
$$\|M_{\alpha}f\|_{L^{q,\infty}(w_{\theta}^q)}\lesssim \|f\|_{L^p(w_{\theta}^p)}$$
but 
$M_{\alpha}$ is {\bf not} bounded from $L^p(w_{\theta}^p)$ into $L^{q}(w_{\theta}^q)$. This can be seen by considering the function $f=\chi_{B_\mathcal{H}(0,1)}$ and using that $M_{\alpha}f(x)\gtrsim w_{\theta}^{q+p'}(x)$.

\item  Let $1< p <\frac{n}{\alpha}$ and $\frac1{q}=\frac1{p}-\frac{\alpha}{n}$. If $\frac1{2}< \theta  <\min\left\{p, p+\frac{q}{p}-\frac{p}{1-\frac{\alpha}{n}}\right\}$ then 
$$\|M_{\alpha}f\|_{L^{q}(w_{\theta}^q)}\lesssim \|f\|_{L^p(w_{\theta}^p)}$$
and the weight $w_{\theta}\not\in A_{p,q}(\mathcal{H}^n)$, since 
$$
\sup_{r(B)>0}  \left(\frac1{\mu_n(B)}\int_B w_{\theta}^q \,d\mu_n \right) \left(\frac1{\mu_n(B)}\int_B w_{\theta}^{-p'} \,d\mu_n \right)^{\frac{q}{p'}}=\infty.
$$
Hence, the class of weight that the fractional maximal operator, $M_{\alpha}$, is bounded from $L^p(w^p)$ into $L^q(w^q)$ is not included in the classical $A_{p,q}(\mathcal{H}^n)$.

\end{enumerate}
\end{ejem}


\end{document}